\documentclass{amsart}
\usepackage[dvips]{color}
\usepackage{graphicx}
\usepackage{amsfonts,graphicx,epsfig, tikz}
\usepackage{latexsym}
\usepackage{amssymb}
\usepackage{amsmath}
\usepackage{amsthm}
\usepackage{amscd}

\usepackage[all]{xy}

\theoremstyle{plain}
\newtheorem{thm}{Theorem}
\newtheorem{lem}{Lemma}
\newtheorem{Bob}{Definition}
\newtheorem{prop}{Proposition}
\newtheorem{cor}{Corollary}

\theoremstyle{remark}
\newtheorem{rem}{Remark}

\newcommand{\RRR}{\ensuremath{\mathcal{R}}}
\newcommand{\ZeroCol}{{\begin{array}{c}0\\\vdots\\0\end{array}}}
\newcommand{\RowOf}[1]{#1\dots #1}
\newcommand{\NICE}{(4321) proxy}
\newcommand{\ANICE}{(4321) quasi-proxy}

\author[J. Chaika]{Jon Chaika}
\author[J. Fickenscher]{Jon Fickenscher}
\begin{document}
\title[Topologically mixing IETs]{Topological mixing for some residual sets of interval exchange transformations}
\begin{abstract} We show that a residual set of non-degenerate IETs on more than 3 letters is topologically mixing. This shows that there exists a uniquely ergodic topologically mixing IET. %This is extended to other Rauzy classes.
 This is then applied to show that some billiard flows in a fixed direction in an L-shaped polygon are topologically mixing.
\end{abstract}
\maketitle
\section{Introduction}
This paper establishes that a residual set of IETs are topologically mixing. In particular, this implies that for a residual set of IETs any large image of an interval is $\epsilon$ dense.

\begin{Bob}
Given $L=(l_1,l_2,...,l_d)$
where $l_i \geq 0$,  we obtain $d$ sub-intervals of $[0,1)$, ${I_1=[0,l_1) },
{I_2=[l_1,l_1+l_2)},...,{I_d=[l_1+...+l_{d-1}, l_1+,...+l_d)}$. Given
 a permutation $\pi$ on $\{1,2,...,d\}$, we obtain a d-\emph{Interval Exchange Transformation} (IET)  $ T_{L,\pi} \colon [0,\sum l_i) \to
 [0,\sum l_i)$ which exchanges the intervals $I_i$ according to $\pi$. That is, if $x \in I_j$ then $$T_{L,\pi}(x)= x - \underset{k<j}{\sum} l_k +\underset{\pi(k')<\pi(j)}{\sum} l_{k'}.$$
\end{Bob}
When there is no cause for confusion the subscript in denoting the IET will be omitted.
Additionally, if $T$ is an IET $\pi(T)$ will denote the permutation of $T$ and $L(T)$ will denote the length vector.

Interval exchange transformations with a fixed permutation on $d$ letters are parametrized by the standard simplex in
 $\mathbb{R}$, $\Delta_{d-1}=\{(l_1,...,l_d): l_i \geq 0, \sum l_i=1\}$ and $\mathring{\Delta}_{d-1}$ denotes its interior, $\{(l_1,...,l_d): l_i > 0, \sum l_i=1\}$.
  The term residual set of IETs will refer to a dense $G_{\delta}$ set of $\Delta_{d-1}$ with respect to the subspace topology.

\begin{Bob} Let $X$ be a topological space. A dynamical system $T: X \to X$ is said to be \emph{topologically mixing} if for any nonempty open sets $U$, $V$ there exists $N_{U,V}:=N$ such that $T^n(U) \cap V \neq \emptyset$ for all $n \geq N$.
\end{Bob}

The following is our main result.
\begin{thm}\label{stronger} A residual set of $d$-IETs in non-degenerate Rauzy classes are topologically mixing for any $d>3$.
\end{thm}
For the definition of non-degenerate see Definition \ref{non deg} on page \pageref{non deg}. The condition on non-degeneracy is
 stronger than what we need. See Remark \ref{remark_3IETs}.
%As Remark \ref{remark_3IETs} in Section \ref{sectionConlusions} indicates, this result actually holds under weaker conditions.

\begin{cor} There exists a uniquely ergodic topologically mixing IET.
\end{cor}
\begin{rem} It was already known that there exists a topologically mixing $4$-IET \cite{top mix}. The IETs given by this construction were non-uniquely ergodic. There exist non uniquely ergodic 5-IETs that are not even topologically weakly mixing (use for example \cite{vskew1}). There are also uniquely ergodic 4-IETs satisfying the Keane condition that are not topologically mixing \cite{hmili}.
No IETs satisfy the stronger property of being measure theoretically mixing \cite{nomixing}.
In fact minimal, uniquely ergodic dynamical systems of linear block growth cannot be measure theoretically mixing \cite{fer}. No 3-IETs can be topologically mixing \cite{bc iet}.
 However, almost every IET that is not of rotation type satisfies the weaker property of being topologically weakly mixing \cite{top w.mix}. In fact, almost every IET that is not of rotation type is measure theoretically weakly mixing \cite{w.mix}.
 This was earlier shown for almost every 3-IET \cite{kat step}, almost every type W \cite{metric} and every type W satisfying an explicit Diophantine condition (Boshernitzan type, which almost every IET satisfies) \cite{bosh nog}.
  For type W IETs it has been shown that the Keane condition implies topologically weakly mixing \cite{cdk}.
\end{rem}
For concreteness we first prove:
\begin{thm}\label{main}
A residual set of 4-IETs in the Rauzy class (4321) are topologically mixing.
\end{thm}
For the definition of Rauzy class see Definition \ref{rauzy class} in Section \ref{rv ind}.
\begin{Bob} An IET $T$ with discontinuities $\delta_1,...,\delta_{d-1}$ is said to satisfy the \emph{Keane condition} (also called the \emph{infinite distinct orbit condition} or \emph{idoc}) if $\{\delta_1,T \delta_1,...,T^k\delta_1,...\}$, $\{\delta_2,T \delta_2,...\}$,...,$\{\delta_{d-1},T\delta_{d-1},...\}$ are all disjoint infinite sets.
\end{Bob}
It is easy to see that the Keane condition is a $G_{\delta}$ condition.  Moreover it is a residual condition.
This is because if the set $L=\{l_1,...,l_d\}$ is linearly independent over $\mathbb{Q}$ then for any irreducible permutation $\pi$ the IET $T_{L,\pi}$ satisfies the Keane condition.
%\begin{thm} (Hmili) There is an idoc IET in the Rauzy class $(4321)$ which is not topologically weak mixing and therfore not topologically mixing.
%\end{thm}

%\begin{rem} The Veech-Satayeev construction shows that there exist minimal non uniquely ergodic 5-IETs satisfying idoc that are not topologically weakly mixing, let alone topologically mixing.
%\end{rem}

%Lets introduce some terminology. Given an idoc IET $T$. Let $R^k(T)$ be the IET obtained by performing $k$ steps of Rauzy-Veech induction.
%Let $I_i^{(k)}$ denote the $i^{th}$ sub-interval of $R^k(T)$ and $B_i^{k}$ denote the symbolic coding of the points in $I_i^{(k)}$ (under $T$) before first return.
%First some notation. Let $c_{k,i}$  the sum of the entries in the $i^{th}$ column of the $k^{th}$ Rauzy-Veech matrix.
The proof of Theorem \ref{main} relies on the next notion which uses the symbolic coding of an IET (see Section \ref{symb}).
\begin{Bob} An IET is called \emph{$k$-alphabet mixing} if there exists a constant $N$ such that for any two allowed blocks of length $k$, $(u_1,...,u_k)$, $(v_1,...v_k)$ and all $n>N$ there exists an allowed block of length $n$ of the form $(u_1,...,u_k,*....*,v_1,...,v_k)$.
\end{Bob}
%The argument is based on showing that for each $k$, a dense open set of IETs has its idoc elements $k$-alphabet mixing. Since idoc is a residual condition it will follow that:
\begin{thm} \label{alphabet} Given a permutation $\pi$ in the Rauzy class of (4321) and an open set $U \subset \Delta_3$ there exists an open set $V\subset U$ so that if ${L}\in V$ and $T_{{L},\pi}$ is idoc then it is $k$-alphabet mixing.
% For each $k$ there exists a dense open set of IETs, $U_k$ such that any idoc element of $U_k$ is $k$-alphabet mixing.
\end{thm}
\subsection{Outline of the argument}
The argument is based on the following: given any start of Rauzy induction on an IET we may prescribe a finite continuation (depending on the start) that guarantees $k$-alphabet mixing for all idoc IETs sharing this beginning.
This shows that in every open set there exists an open set whose idoc elements are $k$-alphabet mixing.
This proof uses Keane's construction \cite{nonue} heavily. We suppress this dependence and instead use the well known Rauzy induction for accessibility reasons.

To show $k$-alphabet mixing for fixed $k$, the main tool is the following fact: if two numbers are relatively prime then any sufficiently large natural number is a  natural number combination of them. To exploit this, we use that a residual set of IETs are exceedingly close to being periodic and so contain many repetitions of a given block. In the argument we use that there are many repetitions of two different blocks, whose lengths are relatively prime. A technical problem is overcome by the fact that for IETs satisfying the Keane condition, there are special words that agree arbitrarily far into the past and differ in the future.

%Our proof is based on combining three main ideas.
%\begin{enumerate}
%\item If two numbers are relatively prime then any sufficiently large natural number is a  natural number combination of them.
%\item A residual set of IETs are exceedingly close to being periodic.
%\item There are special words that agree arbitrarily far into the past and differ in the future.
%\end{enumerate}
\subsection{Structure of the paper}
Section \ref{prelim} collects known results that will be used in the paper: Rauzy induction, Keane's construction and the symbolic coding of an IET. Sections \ref{setup} and \ref{stacked blocks} prove Theorem \ref{main}. Section \ref{other classes} proves Theorem \ref{stronger} with input from the appendix which proves some needed results about non-degenerate Rauzy classes on 4 or more letters. Section \ref{billiard} shows that there is a topologically mixing billiard flow in a fixed direction.
%Idea 3 is where the Keane condition is used and it is partially contradictory to idea 2.
\subsection{Acknowledgments}
J. Chaika thanks M. Boshernitzan and A. Bufetov for their encouragement and Princeton University for its hospitality. J. Chaika was supported by CODY and NSF grants DMS-1004372, DMS-130550. J. Fickenscher thanks Princeton University and his family for their respective and enduring support. J. Chaika and J. Fickenscher thank the referee for their helpful suggestions.
\section{Preliminaries}\label{prelim}

This section collects well known results and establishes terminology.
\subsection{Rauzy induction}\label{rv ind}
Our treatment of Rauzy induction will be the same as in \cite[Section 7]{gauss}.  We recall it here. %The statements here either appear in that section or are immediate consequences that are well known.
 Let $T$ be a $d$-IET with permutation $\pi$. Let $\delta_+(T)$ be the rightmost discontinuity of $T$ and $\delta_-(T)$ be the rightmost discontinuity of $T^{-1}$.
  Let $\delta_{max}(T)=\max\{\delta_+(T),\delta_-(T)\}$. Let $I^{(1)}(T)= [0,\delta_{max}(T))$.
 Consider the induced map of $T$ on $[0,\delta_{\max})$ denoted $T|_{[0,\delta_{\max})}$. If $\delta_+ \neq \delta_-$ this is a $d$-IET on a smaller interval, perhaps with a different permutation.
 We will often write $\delta_-,\delta_+, \delta_{max}, I^{(1)}$ for $\delta_-(T),\delta_+(T), \delta_{max}(T), I^{(1)}(T)$ when there is no confusion.

%We can renormalize it so that it is once again a $d$-IET on $[0,1)$. That is, let $R(T)(x)= T|_{[0,\delta_{\max})}(x\delta_{max}) (\delta_{max})^{-1}$.
 %This is the Rauzy-Veech induction of $T$. To be explicit the Rauzy-Veech induction map is only defined if $\delta_+ \neq \delta_-$.
 If $\delta_{max}= \delta_+$ we say the first step in Rauzy induction is $a$. In this case the permutation of $R(T)$ is given by
\begin{equation} \pi'(j)= \begin{cases}
 \pi (j) & \quad j \leq \pi^{-1}(d)\\ \pi(d) & \quad j=\pi^{-1}(d)+1 \\ \pi(j-1) & \quad \text{otherwise}

\end{cases}.
\end{equation}
We keep track of what has happened under Rauzy induction by a matrix $M(T,1)$ where
\begin{equation} M(T,1)[ij]= \begin{cases} \delta_{i,j} & \quad j \leq \pi^{-1}(d)\\
 \delta_{i, j-1} & \quad j>\pi^{-1}(d) \text{ and } i \neq d\\
\delta_{\pi^{-1}(d)+1,j} & \quad i=d \end{cases}.
 \end{equation}
 If $\delta_{max}= \delta_-$ we say the first step in Rauzy induction is $b$.
 In this case the permutation of $R(T)$ is given by
\begin{equation} \pi'(j)= \begin{cases}
 \pi (j) & \quad \pi(j) \leq \pi(d)\\ \pi(j)+1 & \quad \pi(d) < \pi(j) < d \\ \pi(d)+1 & \quad \pi (j)=d

\end{cases}.
\end{equation}
 We keep track of what has happened under Rauzy induction by a matrix \begin{equation}M(T,1)[ij]= \begin{cases} 1 & \quad i=d \text{ and }j= \pi^{-1}(d) \\ \delta_{i,j} & \quad \text{ otherwise} \end{cases}.
\end{equation}
The matrices described above depend on whether the step is $a$ or $b$ and the permutation $\pi$. The following well known lemmas, which are immediate calculations, help motivate the definition of $M(T,1)$.
\begin{lem} \label{one step} If $R(T)=S_{L, \pi}$  then the length vector of $T$ is a scalar multiple of $M(T,1)L$. \end{lem}
%This is \cite[7.4 and 7.5]{gauss}.
 Let $M_{\Delta}=M\mathbb{R}_d^+ \cap \mathring{\Delta}_d$. Recall $\mathring{\Delta}_d$ is the interior of the simplex in $\mathbb{R}^d$.
\begin{lem}\label{region} An IET with lengths contained in $M(T,1)_{\Delta}$
 and permutation $\pi$ has the same first step of Rauzy induction as $T$.
 \end{lem}

We define the $n^{\text{th}}$ matrix of Rauzy induction by $$M(T,n)=M(T,n-1)M(R^{n-1}(T),1).$$
Likewise, we define $I^{(n)}(T):=I^{(1)}(R^{n-1}(T))$. We will often denote this by $I^{(n)}$.
 %It follows from Lemma \ref{region} that for an IET with length vector in $M(T,n)_{\Delta}$ and permutation $\pi$ the first $n$ steps of Rauzy-Veech induction agree with $T$.
 %If $M$ is any matrix, $C_i(M)$ denotes the $i^{th}$ column and $C_{max}(M)$ denotes the column with the largest sum of entries.
 Let $|C_i(M)|$ denote the sum of the entries in the $i^{th}$ column. Versions of the following lemma are well known and we provide a proof for completeness.
\begin{lem} \label{proscribed rv} If $M(R^n(T),k)$ is a positive matrix and $L= \frac {C_i(M(T,n+k))}{|C_i(M(T,n+k))|}$ then $S_{L, \pi(T)}$ %with lengths $\frac{C_i(M(T,n+k))}{|C_i(M(T,n+k))|}$
 agrees with $T$ through the first $n$ steps of Rauzy induction.
\end{lem}
\begin{proof} By Lemma \ref{one step} the length vector for $R^m(S_{L, \pi})$ is $ \frac{C_i(M(R^m(T),n+k-m))}{|C_i(M(R^m(T),n+k-m))|}$ for any $m$ where $R^m(S_{L,\pi})$ is defined.
  By our assumption on the positivity of $M(R^n(T),k)$ the vector $\frac{C_i(M(R^n(T),k))}{|C_i(M(R^n(T),k))|}$ is contained in $\mathring{\Delta}_d$. The lemma follows by  Lemma \ref{region} and induction.
\end{proof}
\begin{Bob}\label{rauzy class} Given a permutation $\pi_0$ its \emph{Rauzy class} is $$\{\pi': \exists n \in \mathbb{N}, T \text{ with } \pi(T)=\pi_0 \text{ and }\pi(R^n(T))=\pi'\}.$$
\end{Bob}
%The next definition does not appear in \cite{gauss} but is important for the next section.
%\begin{Bob}
 %A matrix $M$ is called $\nu$ \emph{balanced} if $\frac 1 {\nu} <\frac {|C_i(M)|}{|C_j(M)|}<\nu$ for all $i$ and $j$.
%\end{Bob}
 %Notice that if $M$ is $\nu $~ balanced then $|C_i(M)|>\frac {|C_{max}(M)|}{\nu}$.

 \subsection{Keane type induction}\label{keane ind}

 \begin{figure}[t]
	$$ \xymatrix{
	     (2413) \ar@<1ex>[r]^{b}\ar@(ul,ur)^{a}& (2431)\ar[d]_{a}\ar@<1ex>[l]^{b}& (4321)\ar[r]^{b}\ar[l]_{a}& (4132)\ar[d]^{b}\ar@<-1ex>[r]_{a}& (3142)\ar@<-1ex>[l]_{a}\ar@(ur,ul)_{b}\\ 
		 &(3241)\ar[ur]_{a}\ar@(dl,dr)_{b}&   & (4213)\ar[ul]^{b}\ar@(dr,dl)^{a}& }$$
	\caption{The Rauzy Class for $(4321)$.}\label{fig_R4}
\end{figure}
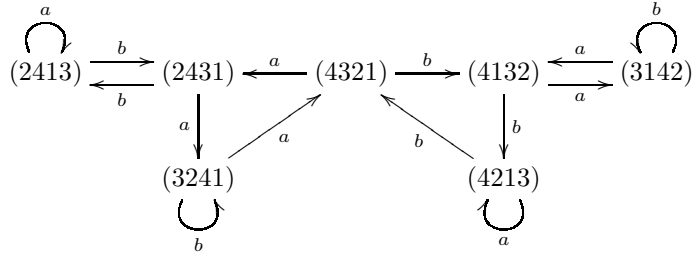

 Consider $(4213)$ IETs. We perform the following sequence of Rauzy induction steps:
 The first $n$ steps are $a$. The next step is b. The next step is $a$. The next step is $b$. Then the next $m$ steps are $a$. The next step is $b$.

  This a total of  $n+m+4$ Rauzy induction steps. %The resulting matrix is
%b once, a once, b once, a m-1 times then b once.
The resulting permutation is $(2431)$. The matrix is:

\begin{center} $ M_1(m,n)=
\left(
\begin{array}{cccc}
1 & 1 & 0 &0 \\
0 & 0 & m+1 & m \\
n+1 & n & n+1 & n+1 \\
1 & 1 & 1 & 1 \end{array} \right)$.%  \left( \begin{array}{c} a \\ b \\ c \\ d \end{array} \right)=\left( \begin{array}{c} c+d \\ (m-1)a+mb \\ n(a+b+ c +d)-c\\ a+b+c+d \end{array} \right)$
\end{center}
%\marginpar{I made it a two row figure rather than\\ three for space.\\ Should I make it\\ \emph{exactly} like the\\ others? I mean this looks very little like a bow-tie now...}
We call this $(m,n)$ Keane induction of the first kind.

 Consider $(2431)$ IETs. The first step is $a$. The next $m$ steps are $b$. The next step is $a$.
 The next $3n+2$ steps are $b$. The next step is $a$.
 This is a total of $m+3n+6$ Rauzy induction steps.
 The resulting permutation is $(4213)$.
 The matrix is:

\begin{center} $M_2(m,n)=
\left(
\begin{array}{cccc}
1 & 1 & 1 &1 \\
n+1 & n+1 & n & n+1 \\
m & m+1 & 0 & 0 \\
0 & 0 & 1 & 1 \end{array} \right) $% \left( \begin{array}{c} a \\ b \\ c \\ d \end{array} \right)=\left( \begin{array}{c} c+d \\ (m-1)a+mb \\ n(a+b+ c +d)-c\\ a+b+c+d \end{array} \right)$
\end{center}

We call this $(m,n)$ Keane induction of the second kind.

We will take advantage of these prescribed paths in the following arguments.
\begin{rem} The reader familiar with Keane's construction will notice that while our matrices look similar to the matrices $A_{m,n}$ in \cite{nonue}, they are different.
 The reason for this difference is that he renames intervals at each step of Keane induction.
 Also, $m$ and $n$ in his paper are $m+1$ and $n+1$ here.
\end{rem}

\subsection{ Symbolic coding of an IET} \label{symb}
We use the symbolic coding of interval exchange transformations heavily.  Throughout this subsection we assume that our IET satisfies the Keane condition.
This section also shows the well known and useful fact that IETs are basically the same as (measure conjugate to) continuous maps on compact metric spaces. %To each point
 \begin{center} Let $\tau \colon [0,1) \to \{1,2,...,d\}^{\mathbb{Z}}$ by $\tau(x)=...,a_{-1},a_0,a_1,...$ where $T^i(x) \in I_{a_i}$. \end{center}
 This map is not continuous as a map from $[0,1)$ with the standard topology to $\{1,2,...,d\}^{\mathbb{Z}}$ with the product topology.
  Observe that the left shift acts continuously on $\tau([0,1))\subset \{1,2,...,d\}^{\mathbb{Z}}$.
  However, if $T$ satisfies the Keane condition then $\tau([0,1))$ is not closed in $\{1,2,...,d\}^{\mathbb{Z}}$ with the product topology.
 This is because the points immediately to the left of a discontinuity give finite blocks that do not converge to an infinite block. Let $ \hat X$ be the closure of  $\tau([0,1))$  in $\{1,2,...,d\}^{\mathbb{Z}}$ with the product topology.
 $\hat X$ corresponds to adding a countable number of points, the left hand side of points in the orbit of a discontinuity and $\hat X$ is a compact metric space.
 Let $f: \hat X \to [0,1)$ be defined by $f|_{\tau([0,1))}= \tau^{-1}$ and extend $f$ by continuity to the rest of $\hat X.$ Notice that unlike $\tau$, $f$ is continuous.
 Moreover the map is injective away from the orbit of discontinuities, where it is $2$ to $1$.
The left shift acts continuously  on $\hat X$ and if $T$ satisfies the Keane condition then the action of the left shift on $\hat X$ is measure conjugate to the action of  $T$ on $[0,1)$.

Fixing a point $x$, that is not in the orbit of a discontinuity, let
\begin{center} $w_{p,q}(x)=c_p,c_{p+1},...,c_{q-1},c_q$ where $\tau(x)=...c_{-1},c_{0},c_1,...$ \end{center}
If $x$ is in the orbit of a discontinuity let $w_{p,q}(x^+)=\underset{y \to x^+}{\lim} w_{p,q}(y)$. Let  $w_{p,q}(x^-)=\underset{y \to x^-}{\lim} w_{p,q}(y)$.
Observe that if $T$ satisfies the Keane condition, $p>0$ and $w_{1,p}(x^+) \neq w_{1,p}(x^-)$ then $ w_{-N,-1}(x^+)=w_{-N,-1}(x^-)$ for all $N>0$.
Let $\mathcal{B}_l(T)=\{a_1,...,a_l: \underset{i=1}{\overset{l} {\cap}} T^{-i}(I_{a_i}) \neq \emptyset\}.$ This is often called the set of allowed $l$ blocks.

The induced map of an IET is a tool for understanding $\mathcal{B}_l(T)$. Consider $I^{(n)}(T):=I^{(n)}$. Each point in $I^{(n)}$ has one of at most $d$ different codings before first return.
%There are only $d$ different codings that $x \in I^{(n)}$ takes under $T$ before first return.
 Represent each possibility by $B_1,B_2,..., B_d$. For any $x \in I^{(n)}$, $\tau(x)$ is a concatenation of these blocks, starting at zero.
 Moreover for any $x$, $\tau(x)$ is a concatenation of these $d$ blocks from a different starting point.
 Therefore $\mathcal{B}_l(T)$ is formed from truncations of concatenations of these blocks.
  Moreover, if $S$ has length vector $L(S)\in M(T,n)_{\Delta}$ and $\pi(S)=\pi(T)$ then $\mathcal{B}_l(S)$ is also formed by truncations of concatenations of $B_1,B_2,...,B_d$.
  This is not to  say that $\mathcal{B}_l(S)=\mathcal{B}_l(T)$ because the blocks $B_1,...,B_d$ can be concatenated in different ways.

  For convenience we introduce some notation. Given a finite word $u$ let $|u|$ denote its length. If $v$ is another finite word let $u*v$ denote the concatenation of these two words.
   So in particular $|u*v|=|u|+|v|$.  Let $u^r$ denote $u$ concatenated to itself $r$ consecutive times.
    Let $B_{j,n}(T)$ be the word formed by the travel of $I_j^{(n)}(T)$ under $T$ until first return to $I^{(n)}$.
     So, $|B_{j,n}(T)|=|C_j(M(T,n))|$. Note that $w_{0,|B_{j,n}|-1}(x)=B_{j,n}(T)$ for every $x \in I_j^{(n)}$.
\section{Setup}\label{setup}
\subsection{Basic combinatorics}\label{basic combinatorics}
\begin{prop} \label{rel prime} Given any $M(T,n):=M$ there exists $U \subset M_{\Delta}$ and $r\in\mathbb{N}$ such that any $S$ with length vector $L(S) \in U$ and $\pi(S) = \pi(T)$ has $\pi (R^{n+r}(S))=(4213)$ and $\gcd(|C_2|,|C_3|)=1$.
\end{prop}

%\begin{prop} \label{rel prime} Given any Rauzy-Veech matrix there is a Keane type continuation such that gcd$(c_{k,2},c_{k,3})=1$
%\end{prop}
The following lemma is used in the proof of the proposition. It is well known and its proof is included for the convenience of the reader.
\begin{lem} \label{sl prime} Let $M$ be a matrix in $SL(d,\mathbb{Z})$ and $c_1,...,c_d$ be the sum of the entries in its $d$ columns. Then $\gcd(c_1,c_2,...,c_d)=1$.
\end{lem}
\begin{proof} $M^{-1} \in SL(d, \mathbb{Z}).$
Let $A$ denote the matrix with all entries 1. $
AM$ has every entry divisible by $\gcd(c_1,c_2,...,c_d)$.
 Therefore, $AMB$ where $B \in SL(d, \mathbb{Z})$ has every entry divisible by $\gcd(c_1,c_2,...,c_d)$.
However $AMM^{-1}=A$. Therefore, $\gcd(c_1,c_2,...,c_d)=1$.
\end{proof}

\begin{proof}[Proof of Proposition \ref{rel prime}]

Without loss of generality we may assume that $\pi(R^nT)=(2431)$.
  Consider $M(T,n)M_2(a,b):=N$. Let  $c_i=|C_i(M(T,n))|$.
Notice $|C_2(N)|=c_1+(b+1)c_2+(a+1)c_3$.
 Also notice that $|C_3(N)|=c_1+ bc_2+c_4$.

The key fact used in this proof is Dirichlet's Theorem which says that the set $\{m, m+n,m +2n,....\}$
contains infinitely many numbers of the form $p\gcd(m,n)$ where $p$ is prime.
 Therefore we can choose $b$ such that $|C_3(N)|=p_1\gcd(c_2, c_1+ c_4)$.
 This is because by varying $b$ the possible values $|C_3(N)|$ can take
 form an arithmetic progression starting from $c_1+c_4$ with increment $ c_2$.
  Likewise, we can choose $a$ such that {$|C_2(N)|=p_2\gcd(c_3,(b+1)c_2+c_1)$} where $p_1 \neq p_2$ are both primes.
   We also stipulate that
   \begin{multline}
    \lefteqn{\gcd(p_1, \gcd(c_2, c_1+ c_4))= \gcd(p_1, \gcd(c_3,(b+1)c_2+c_1))=}\\ \gcd(p_2, \gcd(c_2, c_1+ c_4))= \gcd(p_2, \gcd(c_3,(b+1)c_2+c_1))=1.
    \end{multline}
Under these circumstances $\gcd(|C_2(N)|,|C_3(N)|)=1$.

To show this it suffices to show that $\gcd(c_2 ,c_3 ,c_1+c_4,(b+1)c_2+c_1)=1$.
This follows because if $q$ is a prime then by Lemma \ref{sl prime} $q \nmid c_i$ for some $i$.
 If $i=2,3$ clearly $q \nmid \gcd(c_2 ,c_3 ,c_1+c_4,(b+1)c_2+c_1)$.
If $q |  c_2$, $q| c_3$ and $q|c_1$ then $q \nmid c_4$, thus $q \nmid c_1+c_4$.
If $q |c_2$, $q| c_3$ and $q \nmid  c_1$ then $q \nmid (b+1) c_2+c_1$.
\end{proof}

%\begin{lem} If $M \in \mathbb{N}$ and $M>b_{k,2}b_{k,3}$ then there exists $a, c \in \mathbb{N}$ such that $M=ab_{k,2}+c b_{k,3}$.\end{lem}
%\begin{proof} Let $c_1=M (mod b_{k,2})$ and $c_2=M (mod b_{k,3})$. Because $b_{k,2}$ and $b_{k,3}$ are relatively prime there exists $a_1<b{k,2}, a_2<b_{k,3}$ such that $a_1b_{k,3}=M (mod b_{k,2})$ and $a_2 b_{k,2} =M (mod b_{k,3})$. It follows that $a_1b_{k,3}+ a_2b_{k,2}= M (mod b_{k,2}b_{k,3}).$ By our assumption on $M$, $M>a_1b_{k,3}+ a_2b_{k,2}$ and the lemma follows.
%\end{proof}

%Motivated by this lemma let $g=2b_{2,k}b_{3,k}$.

%Let $n_{k+1}=10b_{k,2}$ and $m_{k+1}=10b_{k,3}$.

%Perform a n-1 times:
%note that a sends 4213 to 4213 not b.
%b once, a once, b once, a m-1 times then b once.

%Perform b $n-1$ times, a $1$ time, b 1 time, a 1 time, b m-1 times.

%\begin{prop} Let $T$ be an IET such that $\pi(R^n(T))=(4213)$ and gcd$(C_2(M(T,n)),C_3(M(T,n)))=1$.
%Let $k$ be such that all $v \in \mathcal{B}_k(T)$ $v \in B_2^{(n)}$ and $v \in B_3^{(n)}$. There exist $a,b$ such that
%if $S \in (M(T,n)M_1(a, b))_{\Delta}$ and $\pi(S)=\pi(T)$ then $S$ is $k$-aplphabet mixing.
%h (4213) permutation and gcd$(c_{k,2},c_{k,3})=1$ there is a Keane type continuation such that is appropriately $l$-alphabet mixing.
%\end{prop}
In the arguments that follow we will assume that $T$ is an IET such that $\pi(R^n(T))=(4213)$ and $\gcd(C_2(M(T,n)),C_3(M(T,n)))=1$.
For $1\leq i\leq 4$, let $c_i= |C_i(M(T,n))|$. Let $g=2|c_2| \cdot  |c_3|$.
% Let $\hat c_i=|C_i(M(T,n)M_1(5g,5g ))|$.

\begin{lem} \label{comb}
If $M \in \mathbb{N}$ and $5gc_2+5gc_3-g\geq M \geq g$ then there exists $a, b \in [0,5g]$ such that $M=a c_2+b c_3$.\end{lem}
%\begin{proof} First let us handle the case where $M<5gc_2$.
%Because $c_2$ and $c_3$ are relatively prime there exists $a_1<c_3, a_2<c_2$ such that $a_1c_3=M \mod c_2$ and $a_2 c_2 =M \mod c_3$.
%It follows that $a_1c_3+ a_2 c_2= M \mod c_2c_3.$
%Let $a:=a_1$ and $b: =a_2+\frac{M-a_1c_3+ a_2c_2}{c_2}$.
% By our assumption on $M$ both $a$ and $b$ are less than $5g$.
%  If $M>5gc_3$ then $5gc_2+5gc_3-M<5gc_2$ and we can handle it similarly.
 % We have proved the lemma if $c_2 \geq c_3$ and the case of $c_3>c_2$ is symmetric.
%\end{proof}
\begin{proof}
	Without loss of generality, we may assume that $c_2\geq c_3$. First, suppose $M\leq 5gc_2$. Because $\gcd(c_2,c_3)=1$, we may choose integers $b_2,b_3$ that satisfy $b_2c_2+b_3c_3=1$. Define $ a_2$ and $0 \leq a_3<c_2$ so that $a_2 = b_2M \mod c_3$ and $a_3 = b_3M\mod c_2$. It follows then that $a_2 c_2 + a_3c_3 = M\mod c_2c_3$. Let $a :=\frac{M-a_3c_3}{c_2}$ and $b:= a_3$. By our assumptions, $0\leq a,b\leq 5g$. If $M>5gc_3$, then $5gc_2+5gc_3 - M < 5gc_2$, and this may be handled similarly.
\end{proof}

To draw a corollary for this lemma let's describe a bit about $\mathcal{B}_r(S)$ for $$S \in (M(T,n)M_1(5g, 5g))_{\Delta} \text{ with } \pi(S)=\pi(T).$$ Let $J:=I^{(n+10g+4)}(S)$.
\subsection{Structure of blocks used to prove k-alphabet mixing}
\label{assumptions}
 Here we show the composition of the $4$
different blocks formed by the induced map on $J$. We will write
these in terms of blocks from the induced map on $I^{n}(S)$.
 To see the number of occurrences of each block we examine the columns of $M_1(5g,5g)$.
 To see the order they appear in recall the Rauzy induction  procedure leading to this matrix.

$\hat B _i:=B_{i,n+10g+4}(S).$

$\hat B_1=B_{1,n}*B_{3,n}^{5g+1}*B_{4,n}$,

$\hat B_2= B_{1,n}*B_{3,n}^{5g}*B_{4,n}$,

$\hat B_3=B_{2,n}^{5g+1}*B_{3,n}^{5g+1}*B_{4,n}$ and lastly

$\hat B_4= B_{2,n}^{5g} *B_{3,n}^{5g+1}*B_{4,n}$.

For ease of notation $B_i$ will denote $B_{i,n}$ for the remainder of the paper.

 All allowed blocks for $S$ come from concatenating these 4 blocks together.
 \begin{lem}\label{size of blocks} $|B_2| \geq |B_1|$ and $2|B_3| \geq |B_4|$.
 \end{lem}
 \begin{proof} This is a consequence of the definition of $M_2(a,b)$, because term by term
  the second column is at least as large as the first and twice the third is at least as large as the fourth.
 \end{proof}
 \begin{lem}$|\hat B_3|\geq |\hat B_i|$ for all $i$. 
 \end{lem}
 This follows from examining the entries of  $M_2(a,b)M_1(5g,5g)$.

\section{Proof of Theorem \ref {main}}\label{stacked blocks}
\subsection{More combinatorics} \label{technical}
This subsection shows that for any $u,v\subset B_2,B_3$ there are large consecutive intervals of natural numbers such that $u*w*v$ is an allowed block.

In proving this we use the fact that the blocks $\hat B_i$ are mostly made up of large repeated blocks of $B_2$ and $B_3$ and $\gcd(|B_2|,|B_3|)=1$.
We treat infinite words as concatenations of $\hat B_i$, examining what happens for one such block at a time.
\begin{lem} Consider an allowed block of the form $B_{2}^{5g}WB_{3}^{5g}$.
%Let $u$ be a subblock of $B_{2,n}$ and $v$ is a subblock of $B_{3,n}$.
For all $$r \in [|W|+g, |W|+ 5g|B_2|+5g|B_3|-g]$$ there exists $w \subset B_{2}^{5g}WB_{3}^{5g}$ such that $|w|=r$ and $w=B_2*m*B_3$.
\end{lem}
\begin{proof} By Lemma \ref{comb} there exists $0\leq a,c \leq 5g$ such that $a|B_{2}|+c|B_{3}|+|W|=r$. Consider the block $B_{2}^aWB_{3}^c$.
\end{proof}
We obtain the following immediate corollary.
\begin{cor}\label{cover rest} If $u \subset B_2$ and $v \subset B_3$ then for any $$r \in [|W|+g +|B_2|+|B_3|, |W|+5g|B_2|+5g|B_3|-g- |B_2|-|B_3|]$$
 there exists $w \subset  B_{2}^{5g}WB_{3}^{5g}$ such that $|w|=r$ and $w=u*m*v$.
\end{cor}
Analogously,
\begin{cor} If $u \subset B_3$ and $v \subset B_2$ then for any $$r \in [|W|+g+ |B_2|+|B_3|, |W|+5g|B_2|+5g|B_3|-g-|B_2|-|B_3|]$$
 there exists $w \subset  B_{3}^{5g}WB_{2}^{5g}$ such that $|w|=r$ and $w=u*m*v$.
\end{cor}
We now apply these results to the situation we are interested in.
\begin{lem}\label{cover 3} For any $u \subset B_3$ and $v \subset B_2$ and
$$r\in [|W|+|B_4|+g +|B_2|+|B_3|, |W|+5g|B_2|+5g|B_3| -|B_2|-|B_3|-g]$$
there exists $w\subset \hat B_i*W*\hat B_3$ such that $|w|=r$ and $w=u*m*v$.
\end{lem}
\begin{proof} Recall from above that the end of any $\hat B_i$ is $B_3^{5g}B_4$. Also $\hat B_3$ starts with $B_2^{5g+1}$. The lemma follows by applying the above results.
\end{proof}
%\begin{rem}\label{cov}
%Observe that if $i=3,4$ then $5g|B_2|+5g|B_3| -|B_2|-|B_3|>|B_i|+ g+|B_2|+|B_3|$.
%\end{rem}
%\begin{lem}\label{cover 2 end} For any $u \subset B_3$ and $v \subset B_2$ and

%\begin{multline}
%\lefteqn{ r\in [|W|+|B_4|+ (5g+1)|B_2|+5g|B_3|+g+ |B_2|+|B_3|,}\\ |W|+|B_4|-g-|B_2|-|B_3|+ 10g|B_2|+10g|B_3|]
%\end{multline}
%there exists $w\subset \hat B_i*W*\hat B_2$ such that $|w|=r$ and $=u*m*v$, when $i=1,2$.
%\end{lem}
%\begin{proof} This follows by considering these blocks as $B_2^{5g}B_3^{5g+1}B_4WB_2^{5g+1}B_3^{5g+1}$
%and applying the above lemmas with $W'=B_3^{5g+1}B_4WB_2^{5g+1}$.
%\end{proof}

%To sum up these lemmas blocks $\hat B_{i_1} \hat B_{i_2}...\hat B_2$ satisfy the hypothesis for $k$-alphabet mixing except for a $g$ neighborhood of when $i_j=1,2$.
%To complement these results let us consider what happens for blocks starting at $\hat B_i$ for $i=1,2$ and ending at $\hat B_3$.
%This is used in the proof of Proposition \ref{covering N}.
\begin{lem}\label{near 2} For any $u \subset B_2$ and $v \subset B_3$ and $r\in$
\begin{multline*}\lefteqn{[|W'|+|B_4|+|B_1|+g+ |B_2|+|B_3|+ 5g|B_3|,}\\ |W'|+|B_4|-g-|B_2|-|B_3|+ 5g|B_2|+10g|B_3|]\end{multline*}
there exists $w\subset \hat B_i*W'*\hat B_2$ such that $|w|=r$ and $w=u*m*v$, when $i=3,4$.
\end{lem}
This follows from Corollary \ref{cover rest} with $W=B_{3}^{5g+1}*B_4*B_1$.

Let $\hat {\delta}$ denote the discontinuity between $I_2^{(n+a+b+4)}(S)$ and $I_3^{(n+a+b+4)}(S)$.
Observe that by the Keane condition $w_{-N,-1}(\hat \delta^+)=w_{-N,-1}(\hat \delta^-)$ and that $w_{0,|\hat B_3|}(\hat \delta^+)=\hat B_3$ and  $w_{0,|\hat B_2|}(\hat \delta^-)=\hat B_2$.
\begin{prop} \label{covering N} %\marginpar{Should this be\\ $r>g + 3\underset{j}\max\{|B_j|\}$?\\ We seem to want\\ $r-g-3\underset{j}\max\{|B_j|\}$\\ to be positive in the proof... I could be misunderstanding.}
Let $u,v \subset B_2$ and $B_3$ and $r>g+3\max_j\{|B_j|\}$. Either there is an allowed $r$ block $$u*m*v=w_{p-r+1,p}(\hat{\delta}^-) \text{ where } p< |\hat B_2|$$
 or there is allowed $r$ block $$u*m*v=w_{p-r+1,p}(\hat{\delta}^+) \text{ where }p<|\hat B_3|.$$
\end{prop}

The `or' is not exclusive.
\begin{proof}
 The proof of this proposition comes by viewing words of the form $w_{n,m}(\hat{\delta}^{\pm})$ as concatenations of blocks $\hat B_j$ for suitable $n<m$. Let $\hat{B}_i$ be the unique block that covers position $-r +g+3\underset{j}\max\{|B_j|\}<0$. 
 %\marginpar{How does this\\ explanation look?\\ I also used $j$ instead of $i$, as $i$ is fixed in the argument.}
 %Let $i$ be the $w_{-(r-g-3\max \{|B_i|\}),-(r-g-3\max \{|B_i|\})+1}(\delta)\in \hat{B}_i$.
  So there exists $$k \in [r-(g+3\max_j\{|B_j|\}),r+|\hat{B}_i|-(g+3\max_j\{|B_j|\})]$$ such that $w_{-k,|\hat{B}_2|}(\delta^-)=\hat{B}_iW\hat{B}_2$
  and $w_{-k,|\hat{B}_3|}(\delta^+)=\hat{B}_iW\hat{B}_3$ where $|W|=k-|\hat{B}_i|$. Lemma \ref{cover 3} implies that there exists an allowed block of length $m$, $u*w*v$ in $w_{-k,|\hat{B}_3|}(\delta^+)$ for  all 
  $$m\in [k-|\hat{B}_i|+g+3\max_j\{|B_j|\},k-|\hat{B}_i|+5g(|B_2|+|B_3|)-g-|B_2|-|B_3|].$$
   Since 
   $$r\geq k-|\hat{B}_i|+g+|B_2|+|B_3|+|B_4|$$ and  
   $$|\hat{B}_i|+g+3\max_j\{|B_j|\} \leq 5g(|B_2|+|B_3|)-g-|B_2|-|B_3|$$
    for $i\in \{1,2\}$  the proposition is proved.
   If $i\in\{3,4\}$ and 
   $$r\leq k-|\hat{B}_i|+5g(|B_2|+|B_3|)-g-|B_2|-|B_3|]$$ Lemma \ref{cover 3} also implies that there exists an allowed block of length $r$, $u*w*v$ in $w_{-k,|\hat{B}_3|}(\delta^+)$. If $$r>k-|\hat{B}_i|+ 5g(|B_2|+|B_3|)-g-|B_2|-|B_3|]$$
    then since 
   $$r\leq k+g+3\max_j\{|B_j|\}< k- |\hat{B_i}|+10g|B_3|+5g|B_2|-g-|B_2|-|B_3|$$
    Lemma \ref{near 2} implies that there exists an allowed block of length $r$,
     $u*w*v$ in $w_{-k,|\hat{B}_2|}(\delta^-).$
  \end{proof}
% if $i \in \{1,2\}$ then by Lemma \ref{cover 3} there exists a block of length $r$, $u*w*v$ in $w_{-|\hat{B}_i|,|\hat{B}_3|}(\delta^+)$. This uses that since $i\in \{1,2\}$ we have
% $|\hat{B}_i|+g+3\max\{|B_i|\} \leq 5g(|B_2|+|B_3|-g-|B_2|-|B_3|$.
% Similarly if $i \in \{3,4\}$ and $w_{-(r+g+4\max\{|B_i|\}}$ is in the same $\hat{B_i}$ then Lemma \ref{cover 3} implies that there exists a block of length $r$, $u*w*v$ in $w_{-|\hat{B}_i|,|\hat{B}_3|}(\delta^+).$ Otherwise Lemma \ref{near 2} implies  there exists a block of length $r$, $u*w*v$ in $w_{-|\hat{B}_i|,|\hat{B}_3|}(\delta^-)$.

%The proof of this proposition comes by viewing words as concatenations of blocks $\hat B_i$. Motivated by this let us assume that
%$w_{-\infty, |\hat B_2|}(\hat{\delta}^-)=...\hat B_jW \hat B_2$ where $|W|<r$ and $|\hat B_j|+|W|\geq r$. We now use the preceding Lemmas to treat different cases.

% If $r-|W|<5g|B_2|+5g|B_3|-|B_2|-|B_4|$ then by Lemmas \ref{cover 3} 
% If $j=3,4$ and $r-|W|>5g|B_2|+5g|B_3|-|B_2|-|B_4|$ let us consider the block $\hat B_jW\hat B_2$. By Lemma \ref{near 2} the proposition follows because
% $$r-|W| \leq |\hat B_j|\leq |B_4|-g-|B_2|-|B_3|+ 5g|B_2|+10g|B_3|.$$
% To see this recall that $g=2|B_2||B_3|$ and therefore $(5g-2)|B_3|-g>2|B_2|$.

\subsection{Proof of Alphabet mixing}
\begin{proof}[Proof of Theorem \ref{alphabet}]
Fix $k$ and an IET satisfying the Keane condition  $T$. Choose $m$ such that every allowed $k$ block of $T$ appears in  $B_{j,m}$ for each $j\in \{1,...,4\}$ (see Subsection \ref{symb}).  %It follows that this is true for $M(T,m)_{\Delta}$. Let $S \in M(T,m)_{\Delta}$ and $n \in \mathbb{N}$ so that $\pi(R^{n+m}S)=(4213)$.
  Now we apply Proposition \ref{rel prime} to obtain $S\in M(T,m)_{\Delta}, n$ so that $\gcd(C_2(M(S,m+n)),C_3(M(S,M+n))=1$ and $\pi(R^{m+n}S)=(4213)$. 
We now define 
$$g=2|C_2(M(S,m+n))|\cdot |C_3(M(S,M+n))|$$ and may now apply Proposition \ref{covering N} to prove the theorem. This is because Proposition \ref{covering N} states that any idoc IET in 
$M(S,n+m)M_2(5g,5g)$  is $k-$alphabet mixing, because each $k-$block is in both $B_{i,n}(T)=B_{i,n}(S)$ for $i\in \{2,3\}$. %because each allowed $k$-block satisfies the hypothesis of the proposition by earlier discussion.
\end{proof}

To complete the proof of Theorem \ref{main} the following proposition.

\begin{prop}\label{alpha implies strong} If $T$ satisfies idoc and is $k$-alphabet mixing for all $k$ then $T$ is topologically mixing.
\end{prop}
\begin{proof} Assume $U$ and $V$ are two nonempty open sets. Therefore, they each contain a small interval.
Because $\{T^i(I_j): i \in \mathbb{Z}, j\in \{1, 2,...,d\}\}$ generates the Borel sigma algebra, these small intervals each contain a $k$-block for some number $k$, say $u$ and $v$.
If $T$ is $k$-alphabet mixing then there exists $N$ such that for all $n>N$ there exists an allowed $n$-block that begins with $u$ and ends with $v$.
Because,   $u \subset U$ and $v \subset V$, $T^{n-|v|}(U) \cap V \neq \emptyset$.
\end{proof}

\begin{proof}[Proof of Theorem \ref{main}]
	 By Theorem \ref{alphabet}, the set of $k$-alphabet mixing IETs contains a dense $G_\delta$ set in the space of IETs for each $k$. The intersection of such sets for all $k$ is therefore also contains dense $G_\delta$ set. By the proposition above, this is a set of topologically mixing IETs.
\end{proof}
%\section{Going Further}

\section{Proof of Theorem \ref{stronger}}\label{other classes}
%We do not really need much to prove residual topological mixing in a
%Rauzy class. What we need is that from any beginning one can reach a
%permutation $\pi$ such that $i<j$, $\pi(i)<\pi(j)$,
%$gcd(|C_i|,|C_j|)=1$.

\subsection{Rauzy Classes with a copy of (4321)}\label{subsectionNICE}

\begin{figure}[t]
	$$ \xymatrix{( d_2\cdots d_0,d_1,d_3) \ar@<1ex>[d]^{b}\ar@(ul,ur)^{a}& &(d_1,d_3\cdots d_0,d_2)\ar@<-1ex>[d]_{a}\ar@(ur,ul)_{b} \\
	     (d_2\cdots d_0,d_1,d_3)\ar[d]_{a}\ar@<1ex>[u]^{b}& (d_0\cdots d_1,d_2,d_3)\ar[r]^{b}\ar[l]_{a}& (d_0,d_3\cdots d_1,d_2)\ar[d]^{b}\ar@<-1ex>[u]_{a}\\ 
		 (d_1\cdots d_2,d_0,d_3)\ar[ur]_{a}\ar@(dl,dr)_{b}&   & (d_0,d_2,d_3\cdots d_1)\ar[ul]^{b^{d-3}}\ar@(dr,dl)^{a}}$$
	\caption{A \NICE\ permutation is an embedding of $\RRR_4$ with only one ``bad'' edge. Here $d_a = d-a$.}\label{fig_NICE}
\end{figure}
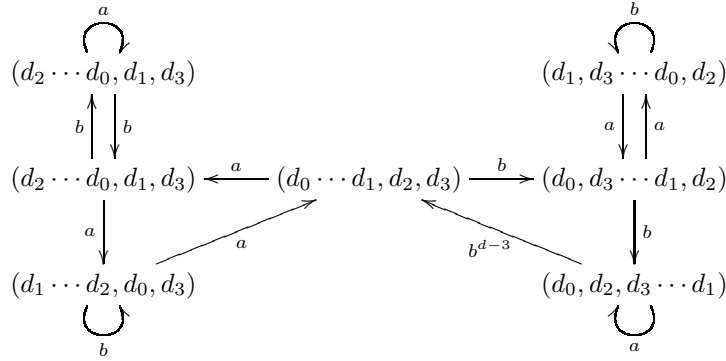

	Let $\RRR_4$ denote the Rauzy class of $(4321)$. We now will describe ways in which $\RRR_4$ can be embedded in classes on more letters.

\begin{Bob}\label{nice_perms}
	 A permutation $\pi$ in a Rauzy class on $d$ letters, $d>3$, will be called a \emph{\NICE} if it satisfies
			$$ \pi(d-3)=d,~\pi(d-2)=d-1,~\pi(d-1)=d-2,~\pi(d)=1.$$
	$\pi$ will be called a \emph{\ANICE} if instead
			$$ \pi(1)=d,~\pi(d-2)=d-1,~\pi(d-1)=d-2,~\pi(d)=1.$$
	 %A Rauzy class will be called \NICE\ (resp. \ANICE) if it contains a \NICE\ (resp. \ANICE) permutation.}
\end{Bob}

%	\marginpar{I thought this\\ helped. Has it\\ just wasted\\ everyone's time?}
	For $d>4$, a \NICE\ or \ANICE\ permutation may be viewed as the permutation $(4321)$ in which all but one or two edges in $\RRR_4$ correspond to the same edge type, as seen by comparing Figure \ref{fig_R4} with Figures \ref{fig_NICE} and \ref{fig_ALMOST_NICE}. These ``bad" edges correspond to a series of edges of the same type in the larger class. Note that when $d=4$, the only \NICE\ or \ANICE\ permutation is $(4321)$ itself.

\begin{prop}\label{classes_nice}
		 Every non-degenerate Rauzy class on $d > 3$ symbols contains either a \NICE\ or a \ANICE\ permutation.
\end{prop}

    \begin{proof}
		From \cite[Corollary 4.1]{fick}, there exists a standard permutation $\sigma$ in a Rauzy class on $N$ symbols such that for some $1\leq \ell \leq d-2$,
			$$ \sigma(j) = 2d-\ell-j-1\mbox{ for } d-\ell\leq j \leq d-1,$$
		or equivalently $\sigma$ places $d-\ell,\dots,d-1$ in reverse order.
		Because $\sigma$ is non-degenerate, $\ell\geq 2$. If $\ell=2$, $\sigma$ is a \ANICE. If $\ell\geq 3$, then by performing $\ell-2$ moves of type $b$ on $\sigma$, we arrive at a \NICE\ permutation.
	\end{proof}

	We may now follow the embedded paths corresponding to matrices $M_1$ and $M_2$ as defined in Section \ref{keane ind} to arrive at the corresponding matrices in these larger classes. We reserve the formulation of the matrices for the appendix. Following the notation in the appendix, we denote the new matrices $\tilde{M}_i$. We will just remark that there is a direct connection between $M_1$ and $M_2$ as defined in Section \ref{keane ind} and these new matrices. In particular, these new matrices contain a copy of their analogues from the $d=4$ case.

	To aid in notation, given $\sigma$ that is either \NICE\ or \ANICE, we denote by $\sigma_\pi$, $\pi\in\RRR_4$, the representative in $\RRR(\sigma)$ in the embedding of $\RRR_4$ centered at $\sigma$ that corresponds to $\pi$. For example if $\sigma = (d \cdots  d-1, d-2,1)$, $\sigma_{(2431)}=(2\cdots d, d-1,1)$ (see Figure \ref{fig_ALMOST_NICE}).

\begin{figure}[t]
	$$ \xymatrix{(2\cdots d,1,d-1) \ar@<1ex>[d]^{b}\ar@(ul,ur)^{a}& &(d-1,1\cdots d,d-2)\ar@<-1ex>[d]_{a}\ar@(ur,ul)_{b} \\
	     (2\cdots d,d-1,1)\ar[d]_{a}\ar@<1ex>[u]^{b}& (d\cdots d-1,d-2,1)\ar[r]^{b}\ar[l]_{a}& (d,1\cdots d-1,d-2)\ar[d]^{b}\ar@<-1ex>[u]_{a}\\ 
		 (3\cdots 2,d,1)\ar[ur]_{a^{d-3}}\ar@(dl,dr)_{b}&   & (d,d-2,1\cdots d-1)\ar[ul]^{b^{d-3}}\ar@(dr,dl)^{a}}$$
	\caption{A \ANICE\ permutation is an embedding of $\RRR_4$ with two ``bad'' edges.}\label{fig_ALMOST_NICE}
\end{figure}
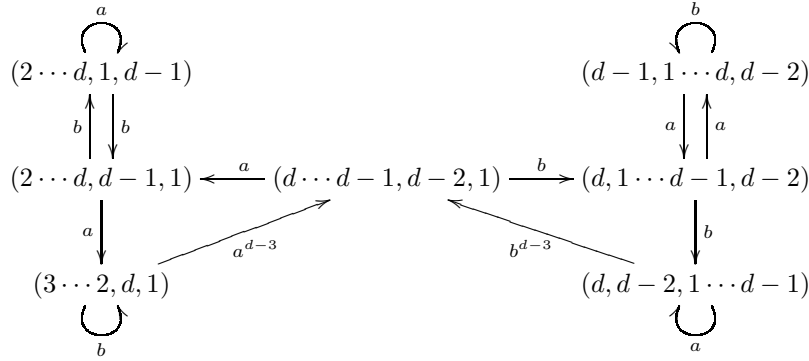

\subsection{Completion of the argument}

  For this section, we make the following assumptions, proven in the appendix. We begin at a $d$-IET $T_{L,\pi}$ for non-degenerate $\pi$. By applying a finite number of Rauzy moves, we arrive at a \NICE\ or \ANICE\ permutation $\sigma$ and by applying one more move of type $a$ and the moves related to $\tilde{M}_2$, we arrive at permutation $\sigma_{(4123)}$ with blocks $B_j$, $1\leq j \leq d$, such that $\gcd(|B_{d-2}|,|B_{d-1}|)=1$. We then perform the moves related to $\tilde{M}_1(m,n)$ to arrive at $\sigma_{(2431)}$ with blocks $\hat{B}_j$, $1\leq j \leq d$. We derive the following in the appendix (compare to the equations before Lemma \ref{size of blocks}):
     \begin{center} \begin{tabular}{|l|l|}
	  \hline $\sigma$ is \NICE & $\sigma$ is \ANICE \\
	  $\hat{B}_\ell = B_\ell B_{d-1}^n B_d,~1\leq \ell \leq d-4$& $\hat{B}_1 = B_1 B_{d-1}^{n+1}B_d$\\
	  $\hat{B}_{d-3} = B_{d-3} B_{d-1}^{n+1} B_d$& $\hat{B}_2 = B_1B_{d-1}^n B_d$\\
	  $\hat{B}_{d-2} = B_{d-3} B_{d-1}^{n+1} B_d$& $\hat{B}_\ell = B_{\ell-1}B_{d-1}^n B_d,~3\leq \ell \leq d-2$\\
	  $\hat{B}_{d-1} = B_{d-2}^{m+1} B_{d-1}^{n+1} B_d$& $\hat{B}_{d-1} = B^{m+1}_{d-2} B^{n+1}_{d-1} B_d$\\
	  $\hat{B}_{d} = B_{d-2}^m B_{d-1}^{n+1} B_d$& $\hat{B}_d = B_{d-2}^{m} B_{d-1}^{n+1} B_d$\\ \hline
       \end{tabular}\end{center}

From these results, we may state the following. This is our generalization of the arguments used in Sections \ref{basic combinatorics} and \ref{assumptions}.

\begin{lem}\label{limited garbage} Given any $m,n>0$ and non-degenerate $\pi$, there exists a finite sequence of Rauzy induction moves ending at $\sigma_{(2431)}$ such that the length between occurrences of either $B_{d-2}^m$ or $B_{d-2}^m B_{d-1}^n$ is at most $C$, where $C>0$ is independent of $m,n$, and $\gcd(|B_{d-2}|,|B_{d-1}|)=1$. 
\end{lem}
The relatively prime condition is Lemma \ref{gen rel prime}. The bound between occurrences is Lemma \ref{C bounded}.
%\marginpar{Should we say something about how big $g$ must be?}
We then proceed as we did in Section \ref{technical} by choosing $m=n = g$ for $g= 2\max\{|B_i|^2\}$.

%\begin{lem}If we are in an almost nice partition's copy of $(4321)$ and $|C_1|$ is prime then there exists a finite path that arrives at $(2431)$ with $gcd(|C_2|, |C_3|)=1$.
%\end{lem}
%\begin{proof} Taking one step to the copy of $(2431)$ we have that $|C_1|$ does not change. Following the path specified for the proof of Proposition 1 and then immediately going to $(4321)$ and the $(2431)$, we arrive at $(2431)$ with $|C_2|=c_1+(b+1)c_2+(a+1)c_3$ and $|C_3|=2c_1+(2b+1)c_2+2c_4$. So $|C_3|=p_1gcd(2c_1+2c_4+c_2,2c_2)$ and $|C_2|=p_2gcd(c_1+(b+1)c_2,c_3)$. By assumption $c_1$ is a prime and so once again by Dirichlet's Theorem $c_1+(b+1)c_2$ and then $|C_2|$ can be taken to be prime. 
%\end{proof}
\begin{lem}Let $\delta$ be the discontinuity that begins coding in the forward direction by $\hat{B}_{d-2}$ on one side (denoted $\delta^*$) and $\hat{B}_{d-1}$ on the other (denoted $\delta^@$). Let $u,v$ be be blocks that appear in both $B_{d-2}$ and $B_{d-1}$.
\begin{enumerate}
 \item If $\omega_{-k-|\hat{B}_{d-2}|,-k}(\delta)=\hat{B}_{d-2}$ then there exists $u*w*v$ a block in $\omega_{-\infty,|\hat{B}_{d-2}|}(\delta^@)$ of length $r$ for all $r\in [k+C+|B_{d-2}|+|B_{d-1}|+g,k+C+5g|B_{d-2}|+5g|B_{d-1}|-|B_{d-2}|-|B_{d-1}|-g].$
 \item If $\omega_{-k-|\hat{B}_{d-1}|,-k}(\delta)=\hat{B}_{d-1}$ then there exists $u*w_r*v$ a block in $\omega_{-\infty,|\hat{B}_{d-1}|}(\delta^@)$ of length $r$ for all 
 $r\in [k+C+|B_{d-2}|+|B_{d-1}|+g,k+C+5g|B_{d-2}|+5g|B_{d-1}|-|B_{d-2}|-|B_{d-1}|-g]$,
 \item  and there exists $u*w_r*v$ a block in $\omega_{-\infty,|\hat{B}_{d-2}|}(\delta^*)$ of length $r$ for all $r \in [(5g+1)|B_{d-1}|+C+g,(5g+1)|B_{d-1}|+5g|B_{d-1}|+(5g+1)|B_{d-2}|+C-g]$.
 \end{enumerate}
  where $C$ is from Lemma \ref{limited garbage}.
\end{lem}
This follows from the lemmas and corollaries in Section \ref{technical}. Items 1 and 2 are Lemma \ref{cover 3}. Item 3 is Corollary \ref{cover rest} with $W=B_3^{5g+1}*B_4*\omega_{-k,-1}(\delta) $.
\begin{cor}\label{some alph mix}  If $T$ is an IET satisfying our assumptions
\footnote{That is, $T$ satisfies idoc, $M(T,m)=M(T,n)\tilde{M}_2(5g,5g)$ where $g$ is large depending on $M(T,n)$ and 
$\gcd(|C_{d-2}(M(T,n))|,|C_{d-1}(M(T,n))|)=1$.} and $u,v$ are blocks that appear in both $B_{d-2}$ and $B_{d-1}$ then there is an allowed word $u*w*v$ of length $r$ for all $r>|B_{d-2}|+|B_{d-1}|+|B_d|+g+C$. $C$ is as in Lemma \ref{limited garbage}.
\end{cor}
\begin{proof}Let $C$ be as in Lemma \ref{limited garbage}. It suffices to show that 
$$|\hat{B}_{d-2}|+g+|B_d|+C<|B_d|+5g|B_{d-2}|+5g|B_{d-1}|-g$$ and 
$$|\hat{B}_{d-1}|+g+|B_d|+C<(5g+1)|B_{d-1}|+5g|B_{d-1}|+(5g+1)|B_{d-2}|-g.$$
 This is a direct computation. By this we cover all times greater than $|B_{d-2}|+|B_{d-1}|+|B_d|+g+C$.
\end{proof} 
\begin{proof}[Proof of Theorem \ref{stronger}] Analogously to the proof of Theorem \ref{alphabet} , Corollary \ref{some alph mix} implies a residual set of IETs is $k$-alphabet mixing for all $k$. These IETs are topologically mixing by Proposition \ref{alpha implies strong}.
\end{proof}
%Using that hyperelliptic has loops (of one winner) with lower order
%hyper elliptic branching out, inductively run through getting the
%set that is left relatively prime. When one gets down to four apply
%the above approach.
\section{A topologically mixing billiard flow in a fixed direction}\label{billiard}

This section treats the flow in 
a given direction of  billiard in an L-shaped polygon. We stress that the acting space is 2-dimensional
(4 copies of the L-shaped table because the angle of travel takes at
most 4 values). %For an introduction to these topics see...

An L-shaped table is determined by 2 horizontal lengths that we will
denote by $a$ and $b$ and two vertical lengths that we denote $s$
and $t$. The vertical cylinders have circumferences $a$ and $b$ and periods
$s+t$ and $s$. When the L-shaped billiard is unfolded following \cite{katzem} one obtains an
 L-shaped table of twice the parameters and opposite sides identified. For this reason we restrict to the sufficient case of the L-shaped table with opposite sides identified.
 
 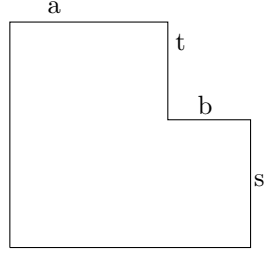
\begin{figure}[ht]
\begin{tikzpicture}
\draw (0,0)--(3.2,0)--(3.2,1.7)--(2.1,1.7)--(2.1,3)--(0,3)--(0,0)
node[ right= 4.75 cm, above=.7 cm, text width=3 cm] { s}
node[ right= 3.7 cm, above=2.5 cm, text width=3 cm]{t}
node[ right= 4 cm, above= 1.65 cm, text width=3 cm] {b} 
node[ right= 2 cm, above=3 cm, text width=3 cm]{a} ;
\end{tikzpicture}
\caption{An L-Shaped table}
\end{figure}

The next three lemmas are known and probably well known.
 Their proofs are straightforward.

\begin{lem} Given an L-shaped table with widths $a,b$ and heights
$s,t$ the flow in direction $\theta$ to the horizontal gives an IET
with lengths $L_1=a-(s+t)\cot(\theta),L_2=t\cot(\theta), L_3=b,
L_4=s\cot(\theta)$ and permutation $(2413)$ so long as the
quantities are all positive.
\end{lem}
\begin{lem}The time to first return is given by
$\frac{s}{\sin(\theta)}$ for $I_3$ and $\frac{s+t}{\sin(\theta)}$
the other intervals.
\end{lem}
In particular, the first return time for $I_3$ is given by
$\frac{|I_4|}{\cos(\theta)}$ and the first return times for the
other intervals is $\frac{|I_2|+|I_4|}{\cos(\theta)}$.
\begin{lem}The set of IETs that arise as the transversal of an
L-shaped table contains an open set in the parameterizing space of IETs
with permutation $(2413)$.
\end{lem}
%There are two difficulties:

%1) We need to address a continuous parameter of times.

%-Fix is to consider the difference in time of represented by a $B_2$
%and a $B_3$ and to cycle through as often as necessary to get
%$k-\epsilon$ mixing.

\begin{thm}The set of IETs that arise as the transversal of a
topologically mixing billiard flow in a fixed direction contains a $G_{\delta}$ set that is dense in an open set of $(2413)$-IETs.
\end{thm}
%\begin{Bob} A billiard flow $F:X \to X$ is called $\epsilon$-mixing if
%for any $x,y$ there exists $N$ such that $F^t(B(x,\epsilon)) \cap
%B(y,\epsilon)\neq \emptyset$ for all $t>N$.
%\end{Bob}

Let $F$ be a minimal rational billiard flow with transversal IET
$T_F$. The billiard flow can be seen as a suspension over the IET
where the heights are constant on intervals. Let $h_i$ be the height
of the rectangle over $I_i(T_F)$.
\begin{Bob} A billiard flow $F$ with transversal IET $T_F$ is called \emph{$(k,\epsilon)$-alphabet mixing}
if for any large enough $t$ and $u$, $v\in \mathcal{B}_k(T_F)$ there
exists an allowed block $w=u*w'*v$ such that
$|\underset{i=1}{\overset{|w|}{\sum}} h_{w_i} -t|< \epsilon$.
\end{Bob}
\begin{prop} If $(F,T_F)$ is $(k,\epsilon)$-alphabet mixing for all $k \in \mathbb{N}$ and $\epsilon>0$ then $F$ is
topologically mixing.
\end{prop}
%\begin{prop} It suffices to show that $T_F$ is $k$-alphabet mixing
%for all $k$ and that for any large enough $t$ and $u$, $v\in
%\mathcal{B}_k(T_f)$ there exists an allowed block $w=u*w'*v$ such
%the $|\underset{i=1}{\overset{|w|}{\sum}} h_{w_i} -t|< \epsilon$.
%\end{prop}
The proof of this proposition is a straightforward % and left as an
exercise left to the reader. %Motivated by this proposition we present the
%following definition.
%\begin{Bob} We say $F_{\theta}$ is $k,\epsilon$-alphabet mixing if for $u,v \in \mathcal{B}_k(T_{\theta})$
% $$\{t: F^{t}_{\theta}(\tau^{-1}(u))\cap \tau^{-1}(v) \neq \emptyset\}$$ is $\epsilon$ dense in $[N,\infty)$ for large enough $N$.
%Where $\tau: h \to \{1,...,d\}^{\mathbb{Z}}$ is the map from Section
%\ref{symb} with respect to $T_{\theta}$.
%\end{Bob}
%\begin{Bob} Let $F$ be a billiard flow arising from an IET $T_F$.
%$(F,T_F)$ are called $\epsilon$-suspension mixing if
%\end{Bob}

%\begin{lem}As $m_{k+1},n_{k+1}$ tend to infinity what we limit on can have
%$gcd(\sum h_iC_2(i),\sum h_iC_3(i)$ small.
%\end{lem}

\begin{lem}Let $\delta>0$. Let $S$ be a $(2413)$-IET satisfying the Keane condition, which arises
 from the transversal of a
billiard in an L-shaped table with slope $\theta$. Assume that
\begin{enumerate}
\item $\pi(R^n(S))=(4213)$
\item any allowed $k$-block of $S$ appears
in $B_{2,n}$ and $B_{3,n}$ and
\item $q=\frac{\gcd(C_3[4],
C_3[2]+C_3[4])}{\cos(\theta)|C_3|}$ where $C_3:=C_3(M(S,n))$.
\end{enumerate}
 Then there
exists an open set contained in $\Delta_{M(S,n)}$  each of whose
idoc elements is the transversal for a billiard in an L-shaped
polygon that is $(k,q+\delta)$-alphabet mixing.
\end{lem}
\begin{proof} Consider $M(S,n)M_1(a,b)_{\Delta}$. As $a,b$ go to infinity and
$\frac a b$ goes to 0, the lengths of the IETs in this set approach
$\frac{C_3}{|C_3|}$. Also they have

$B_{1,n+a+b+4}=B_{1,n}*B_{3,n}^{b+1}*B_{4,n}$

$B_{2,n+a+b+4}=B_{1,n}*B_{3,n}^{b}*B_{4,n}$

$B_{3,n+a+b+4}=B_{2,n}^{a+1}*B_{3,n}^{b+1}*B_{4,n}$

$B_{4,n+a+b+4}=B_{2,n}^a*B_{3,n}^{b+1}*B_{4,n}$.

 Because of the fact that
 $\frac{C_3(M(S,n)M_1(a,b))}{|C_3(M(S,n)M_1(a,b))|}$
 is close to $\frac{C_3}{|C_3|}$, the suspensions
$h_1,h_2,h_4$ are close to
$\frac{(C_3[2]+C_3[4])\tan(\theta)}{\sin(\theta)|C_3|}$ and the
suspension $h_3$ is close to
$\frac{C_3[4]\tan(\theta)}{\sin(\theta)|C_3|}$. By an argument similar to
Section \ref{technical} the proof follows by choosing $a,b$
sufficiently large.
\end{proof}
\begin{lem} For all $\epsilon, \theta>0$ there exists a dense set of
IETs $S$ such that there exist $n_S:=n$ with
 \begin{enumerate}
\item $\pi(R^n(S))=(4213)$
\item $\gcd(C_3[2],C_3[4])=1$
\item $\frac{1}{\cos(\theta)|C_3|}<\epsilon$.
\end{enumerate}
\end{lem}

\begin{proof} The only difficulty is showing that conditions 1 and 2
can be simultaneously satisfied. It is easy to see from \cite[Part II]{metric}, which shows that the IET with permutation $\pi(T)$ and lengths $C_i(M(T,n))$ is primitive, that the entries
of $C_i(M(T,n))$ are relatively prime as a set. The lemma follows
similarly to Proposition \ref{rel prime}. By varying $m$ in
$M_1(m,n)$ it is easy to show that one can have
\begin{multline*}
\gcd(C_3[2],C_3[4])=\\
\gcd(C_3(M(T,r))[3],C_3(M(T,r))[4],C_3(M(T,r))[1]+C_3(M(T,r))[2])
\end{multline*}
when $\pi(R^r(T))=(2413)$. By using $M_2(m,n)$ the proof is
completed.
\end{proof}
\section{Concluding remarks}\label{sectionConlusions}
%While Theorem \ref{stronger} covers non-degenerate permutations, this restriction is only for convenience. In fact, only the third type of degenerate permutation (as given in Definition \ref{non deg}) can prevent the existence of a \NICE\ or \ANICE\ permutation in Proposition \ref{classes_nice}. If we allow Ruazy induction on the left as well as on the right, we may find such permutations in every extended Rauzy class that is not of $2$-IET or $3$-IET type.

\begin{rem}
	We feel that hiding in the background of this proof is the fact that almost every IET that is not of rotation type is weakly mixing.
 An obstruction to appearing mixing occurs in the symbolic coding when we have a word of the form $UWU$.
  At time $r=|W|+|U|$ there is no hope for points in this block to appear mixing.
   Such a circumstance happens for any particular block with bounded gaps (under the assumption of minimality).
   If one looks at a discontinuity one has another choice, $VWU$ for some $V$.
    It seems reasonable to expect that often in the presence of weak mixing these two blocks cover all times. That is, one only needs to look at the coding of the right and left sides of a discontinuity to verify $k$-alphabet mixing. Indeed, our proof shows that this works for a residual set of 4-IETs. 
    3-IETs show that this is not always the case. This is a very special situation, because 3-IETs arise as induced maps of rotations and the absence of topological mixing is a consequence of this.
    In particular, it forces the Rokhlin towers over certain intervals to be closely related.
    In short, we guess that the measure theoretic version of Theorem \ref{stronger} holds: almost every IET in a non-degenerate stratum with more than 3 letters is topologically mixing.
    \end{rem}
    
    \begin{rem}\label{remark_3IETs} The condition that permutations are non-degenerate is actually stronger than is necessary for Theorem \ref{stronger}.
    An IET $T$ is of \emph{$3$-IET type} if it is equal to a $3$-IET $S$, meaning $S(x) = T(x)$ for all $x\in [0,1)$.
    Note that a $2$-IET $T$ is actually of $3$-IET type, as we can define $S$ with $\pi(S) = (312)$ and $L(S) = (l_1,l_2/2,l_2/2)$ where $L = L(T)$ is $L= (l_1,l_2)$.
    We may then restate Theorem \ref{stronger} as ``A residual set of IETs not of $3$-IET type are topologically mixing.''
    We outline the modification to the proof in this paper. We consider \emph{extended Rauzy Classes}, which allow for Rauzy induction on the left as well as on the right.
    This induction from the left has associated matrices $M$ similar to those in Section \ref{prelim}, and any beginning of induction (of both left and right induction) yields open sets with the same blocks as mentioned in Section \ref{symb}.
    We can show that every extended Rauzy class, that does not arise from $3$-IET type transformations, contains a \NICE\ or \ANICE\ permutation (compare to Proposition \ref{classes_nice} in Section \ref{subsectionNICE}). The rest of the arguments are then similar.
    \end{rem}

\appendix
\section{Proof for $d>3$}

  In order to generalize Theorem \ref{main} for $d$-IETs, $d>3$, we must first appeal to established results concerning general Rauzy Classes. We recall that a permutation $\pi$ is standard if $\pi(1)=d$ and $\pi(d)=1$. In the paragraph immediately following the definition of what was later named Rauzy Classes, Rauzy himself proved that every class contains at least one standard permutation (see \cite[Paragraph 35]{rauzy}).

  \begin{Bob}\label{non deg} A permutation is \emph{degenerate} if one of the following conditions are satisfied for some $1\leq j < m$:
  \begin{itemize}
      \item $\displaystyle \pi(j+1) = \pi(j)+1$.
      \item $\displaystyle \pi(j) = m,~\pi(j+1) = 1,\mbox{ and } \pi(1) = \pi(m)+1.$
      \item $\displaystyle \pi(j+1) = 1\mbox{ and }\pi(1)=\pi(j)+1$.
      \item $\displaystyle \pi(j+1) = \pi(m)+1\mbox{ and }\pi(j)=m$.
     \end{itemize}
    A permutation is \emph{non-degenerate} if it is not degenerate.
  \end{Bob}
  See \cite{masur} and \cite{gauss} for original definitions. One may furthermore read in \cite{gauss} that $\pi$ is non-degenerate if and only if every $\pi'$ in the same Rauzy Class is non-degenerate. The results provided in Section \ref{other classes} are stated for non-degenerate Rauzy Classes, meaning every permutation in the class has this property.

%  The classification of all Rauzy Classes as presented in \cite{K-Z} and completed in \cite{boissy} may be used to show that Rauzy Classes are closed under inversion, meaning $\pi\in\RRR$ for Rauzy Class $\RRR$ if and only if $\pi^{-1}\in\RRR$. In \cite{fick}, the second author states this as an explicit result by showing that every class contains a permtuation such that $\pi=\pi^{-1}$ and using the following relationship between inverses and Rauzy induction (see \cite[Equations (2.2)]{metric}):
%      \begin{equation}\label{inverse moves}
%	  (a\pi)^{-1} = b[\pi^{-1}]\mbox{ and } (b\pi)^{-1} = a[\pi^{-1}]
%      \end{equation}
%    for all irreducible $\pi$. Here $a\pi$ (resp. $b\pi$) denotes the resulting permutation after taking a type $a$ (resp. $b$) induction on $\pi$.

\subsection{The general form of $M_1$ and $M_2$}
  Recall the definition of the \NICE\ and \ANICE\ permutations from Definition \ref{nice_perms}. We also use the notation $\sigma_\pi$, where $\sigma$ is \NICE\ or \ANICE\ and $\pi\in\RRR_4$, the class for $(4321)$. As Figures \ref{fig_NICE} and \ref{fig_ALMOST_NICE} indicate, such embeddings enjoy the following properties: any type of Rauzy induction move on $\RRR_4$ represents a move or iteration of moves of the same type in our larger class. To be more precise, $\sigma_{c\pi} = c\sigma_{\pi}$, $\pi\in\RRR_4$ and $c=a,b$, with the following exceptions:
      \begin{itemize}
	\item $\sigma = \sigma_{(4321)} = b^{d-3}\sigma_{(4213)} = \sigma_{b(4213)}$ for \NICE\ or \ANICE\ $\sigma$ , and
	\item $\sigma = \sigma_{(4321)} = a^{d-3}\sigma_{(3241)} = \sigma_{a(3241)}$ for \ANICE\ $\sigma$.
      \end{itemize}
  When describing the paths to generate the new versions of $M_1$ and $M_2$ in each case below, we are following the sequence of moves defined on $\RRR_4$ as given in Section \ref{keane ind}. Note that in all cases, the original four dimensional matrices appear on four columns in their higher dimensional analogues with zeros appearing above and below. This is crucial for our generalization, we work with the four columns that ``look like" our original four columns on $\RRR_4$.

  Consider first a \NICE\ permutation $\sigma$. The path corresponding to $M_1$ begins at $\sigma_{(4213)}=(d,~d-2,~d-3\cdots d-1)$ and ends at $\sigma_{(2431)}=(d-2\cdots d,~d-1,~d-3)$ with matrix
		$$ \tilde{M}_1(m,n)=\left(
				\begin{array}{ccccc}
				  I_{d-4} & \ZeroCol & \ZeroCol & \ZeroCol & \ZeroCol\\
				  \RowOf{0} & 1 & 1 & 0 & 0\\
				  \RowOf{0} & 0 & 0 & m+1 & m\\
				  \RowOf{n} & n+1 & n & n+1 & n+1\\
				  \RowOf{1} & 1 & 1 & 1 &1 
				\end{array}
			\right)$$
  is achieved as follows: take $n$ steps of type $a$, $d-3$ steps of type $b$, a step of type $a$, a step of type $b$, $m$ steps of type $a$ and then a step of type $b$.

    The path corresponding to $M_2$ begins at $\sigma_{(2431)}$ and ends at $\sigma_{(4213)}$ with matrix
		$$ \tilde{M}_2(m,n) = \left(
				\begin{array}{ccccc}
				  I_{d-4} & \ZeroCol & \ZeroCol & \ZeroCol & \ZeroCol\\
				  \RowOf{0} & 1 & 1 & 1 & 1\\
				  \RowOf{n} & n+1 & n+1 & n & n+1\\
				  \RowOf{0} & m & m+1 & 0 & 0\\
				  \RowOf{0} & 0 & 0 & 1 &1 				
				\end{array}
				\right)$$
    and is described by the following sequence of moves: one type $a$ move, $m$ moves of type $b$, one move of type $a$, $n(d-1)+2$ moves of type $b$ and then one move of type $a$.

    Now consider \ANICE\ permutation $\sigma$. The path corresponding to $M_1$ in this case will yield
	    $$ \tilde{M}_1(m,n) = \left( \begin{array}{ccccc}
				    1 & 1 & \RowOf{0} & 0 & 0 \\
				    \ZeroCol & \ZeroCol & I_{d-4} & \ZeroCol & \ZeroCol \\
				    0 & 0 & \RowOf{0} & m+1 & m \\
				    n+1 & n & \RowOf{n} & n+1 & n+1 \\
				    1 & 1 & \RowOf{1} & 1 & 1
	                        \end{array}\right)$$
    with the same sequence of moves as given for the \NICE\ case. Here $\sigma_{(4213)}=(d,d-2,1\cdots d-1)$ and $\sigma_{(2431)} = (2\cdots d,d-1,1)$.

    The path corresponding to $M_2$, which begins at $\sigma_{(2431)}$ and ends at $\sigma_{(4213)}$, is
	    $$ \tilde{M}_2(m,n) = \left(\begin{array}{ccccc} 
		  1 & \RowOf{1} & 1 & 1 & 1 \\
		  n+1 & \RowOf{n} & n+1 & n & n+1 \\
		  \ZeroCol & I_{d-4} & \ZeroCol & \ZeroCol \\
		  m & \RowOf{m} & m+1 & 0 & 0 \\
		  0 & \RowOf{0} & 0 & 1 & 1
		\end{array}\right)$$
    generated by the following sequence of moves: one type $a$ move, $m$ moves of type $b$, $d-3$ moves of type $a$, $n(d-1)+2$ moves of type $b$ and then one move of type $a$.

Note that in an abuse of notation, $\tilde{M}_i(m,n)$ refers to both the \NICE\ and \ANICE\ matrices. 
\subsection{A technical lemma}	

    In order to provide a generalized argument, we must ensure that when we arrive at a \NICE\ (resp. \ANICE) permutation, the columns $C_{d-3}$ (resp. $C_1$), $C_{d-2}$, $C_{d-1}$ and $C_d$ have relatively prime sums. This will replace Lemma \ref{sl prime}, which would only ensure that all $d$ columns have relatively prime sums as a set. We will actually show a stronger result first, which we will use in later arguments.

    \begin{lem}\label{lem.cd_prime}
      Given any $T=T_{L,\pi}$ and standard permutation $\sigma\in\RRR(\pi)$, define $M:=M(T,n)$. Then there exists $r>0$ and $U\subset M_{\Delta}$ such that for all $S$ with length vector $L(S)\in U$ and $\pi(S) = \pi(T)$, $\pi(R^{n+r}(S)) = \sigma$ and $|C_d|=q$, where $q$ is a prime number such that $q\nmid |C_j|$ for any $j\neq d$.
    \end{lem}

%   Regarding the paragraph before Equation \eqref{inverse moves}, we may replace $\sigma$ with $\sigma^{-1}$ in the proof of Lemma \ref{lem.c1_prime}. By returning to $\sigma = [\sigma^{-1}]^{-1}$ after performing our moves, we get the following as a conequence.

%    \begin{cor}\label{cor.cn_prime} The result from lemma \ref{lem.c1_prime} may apply to $|C_d|$ instead.\end{cor}

    Before proving this lemma, let us now state the following:

    \begin{cor}\label{cor.c1_prime_embeddings}
      Given any $T$ as in the previous lemma such that $\pi=\pi(T)$ is non-degenerate, there exists a a permutation $\sigma\in\RRR(\pi)$ that is either \NICE\ or \ANICE\ such that
	$$ G = \gcd(|C_{m}|,|C_{d-2}|,|C_{d-1}|,|C_d|) = 1,$$
      where $m=d-3$ if $\sigma$ is \NICE\ and $m=1$ is $\sigma$ \ANICE.
    \end{cor}

  \begin{proof}[Proof of Corollary \ref{cor.c1_prime_embeddings}]
      By \cite[Corollary 4.1]{fick}, we may find a standard permutation $\sigma'\in\RRR(\pi)$ with $\ell\geq 2$ as defined in the proof of Proposition \ref{classes_nice}. If $\ell=2$, $\sigma$ is \ANICE. By applying Lemma \ref{lem.cd_prime}, we get $|C_d|=q$ that does not divide any other $|C_j|$. In particular, $q\nmid |C_{d-1}|$, so
	$$ G \leq \gcd(|C_{d-1}|,|C_d|) = 1.$$
      If $\ell\geq 3$, we first apply Lemma \ref{lem.cd_prime} to $\sigma'$. We then reach $\sigma$ that is \NICE\ by $\ell-2$ moves of type $b$, where $C'_{d} = C_d$ and $C'_{d-1}= C_{d-1}$. Therefore
      $$ G \leq \gcd(|C'_{d-1}|,|C'_d|) = 1.$$
    In either case, we arrive at a desired $\sigma$ with $G=1$.
  \end{proof}

  \begin{proof}[Proof of Lemma \ref{lem.cd_prime}]
  We first choose any finite path from $\pi$ to $\sigma$ and consider the resulting matrix with columns $C_j$.
  Let $c_j = |C_j|$. We will perform an iterative argument as follows: We express $c_d$ as
	$$ c_d = q D$$
  where $q$ is a prime such that $q\nmid c_j$ for $j\neq d$ or initially $q=1$ with the condition that $c_j\neq c_d$ for $j\neq d$. We then perform a finite closed path of Rauzy induction on $\sigma$ to arrive at a new matrix such that $c'_d = q' D'$ where $q'>q$ satisfies the same condition with respect to $c_j'$, $j\neq d$, and $D'<D$. If we show this, then we may continue this argument until $D=1$, proving the claim.
  
  If our matrix does not satisfy the initial conditions for $q=1$, then we may perform $n-1$ moves of type $b$ to return to $\sigma$ with $c'_d=c_d$ and $c'_j = c_d + c_j$ for $j\neq d$. This new matrix will now be valid under our assumptions, still for $q=1$.

  Let us assume we are at a particular step of the iterative argument, so $c_d = qD$. Because $q\nmid c_j$ for all $j\neq d$, $\gcd(c_d,c_j) \leq D$. Consider any prime $p$ such that $p\mid D$. By Lemma \ref{sl prime}, there must exist $\ell\neq d$ such that $p\nmid c_\ell$. Fix such an $\ell$ and perform the following closed Rauzy path from $\sigma$, $N-\sigma^{-1}\ell$ moves of type $b$, $s(N-\ell)$ moves of type $a$ followed by $\Big(\sigma^{-1}\ell\Big)-1$ moves of type $b$ for some $s>0$. The matrix for this given path is
%\marginpar{This is fixed for $C_d$\\ rather than $C_1$}
		  $$ M_*(s,\ell)_{ij} = \left\{\begin{array}{ll} 1, & i=j\neq \ell \mbox{ or } i=d~\&~j>1 \\
					s+1, & i=j=\ell\\
					s, & i=\ell~\&~j=d\mbox{ or } i=\ell~\&~j<\ell~\&~\sigma^{-1}(j)<\sigma^{-1}(\ell)\\
						& \mbox{or }i=\ell~\&~j>\ell~\&~\sigma^{-1}(j) > \sigma^{-1}(\ell),\\
					2s, & i=\ell~\& ~\ell<j<d ~ \& ~ \sigma^{-1}(j)<\sigma^{-1}(\ell),\\
					0, & \mathrm{otherwise.} \end{array}\right.$$
      We leave the calculation of $M_*$, which follows from the fact that $\sigma$ is standard, as an exercise. In particular $c'_d = c_d + s c_\ell$ and $c'_j = c_d + c_j + s\tau_jc_\ell$ for $j\neq 1$, where
	  $$ \tau_j = \left\{\begin{array}{ll} 2, & j>\ell ~\&~ \sigma^{-1}(j)<\sigma^{-1}(\ell) \\ 0, & j<\ell ~\&~ \sigma^{-1}(j) > \sigma^{-1}(\ell)\\ 1, & \mathrm{otherwise.}\end{array}\right.$$
      By Dirichlet's Theorem, we may choose $s$ large enough that $c'_d = q' \gcd(c_d,c_\ell)$ where $q'>q$ is prime, $D' = \gcd(c_d,c_\ell) < D$ and for all $j\neq d$,
	  $$ q'\nmid c'_j = c_j - (\tau_j-1)c_d + \tau_jq'D'.$$
      We have concluded our iterative step, and the theorem.
  \end{proof}

\subsection{Finishing the argument}

  We may now make the generalized claims using the results from the $\RRR_4$ case.

  \begin{lem} \label{gen rel prime}Given any non-degenerate permutation $\pi$ on $d$ symbols, $d>3$, with \NICE\ or \ANICE\ $\sigma\in\RRR(\pi)$, there exists a finite sequence of Rauzy inductive moves beginning at $\pi$ and ending at $\sigma_{(4213)}$ such that $\gcd(|B_{d-2}|,|B_{d-1}|)=1$.
  \end{lem}

  \begin{proof}
      We consider the type of $\sigma$. Let $c_j = |C_j|$, for $1\leq j \leq d$.

      If $\sigma$ is \NICE, then by Corollary \ref{cor.c1_prime_embeddings}, $\gcd(c_{d-3},c_{d-2},c_{d-1},c_d)=1$.  By performing one move of type $a$, we arrive at $\sigma_{(2431)}$. If $C'_j$ represent the column sums of the matrix after this move, then $|C'_j|=c_j$  for all $j \leq d-3$, $|C'_{d-2}| = c_d + c_{d-3}$, $|C'_{d-1}|=c_{d-2}$ and $|C'_d|=c_{d-1}$. Then
	$$ \gcd(|C'_{d-3}|,|C'_{d-2}|,|C'_{d-1}|,|C'_{d}|) = \gcd(c_{d-3},c_{d-2},c_{d-1},c_d)=1.$$
      By acting on $\sigma_{(2431)}$ by the path corresponding to $\tilde{M}_1(m,n)$ in this case, we arrive at $\sigma_{(4213)}$ with
	  $$ \begin{array}{rcl} |B_{d-2}| &=& |C'_{d-3}| + (n+1)|C'_{d-2}| + (m+1)|C'_{d-1}|\mbox{ and }\\
				|B_{d-1}| &=& |C'_{d-3}| + n|C'_{d-2}| + |C'_{d}|
	     \end{array}$$
      where $B_j$ are the words associated to each interval $I_j$. By repeating the same argument from Lemma \ref{rel prime}, we may choose $m$ and $n$ so that
	  $$\gcd(|B_{d-2}|,|B_{d-1}|)=1.$$

      If instead $\sigma$ is \ANICE, we have that $\gcd(c_1,c_{d-2},c_{d-1},c_d)=1$. Then the move of type $a$ to $\sigma_{(2413)}$ has columns $C'_j$ such that (in particular)
	  $$\gcd(|C'_1|,|C'_2|,|C'_{d-1}|,|C'_d|) = \gcd(c_1, c_1+c_d, c_{d-2}, c_{d-1}) = 1.$$
      It follows that again by following the path for $\tilde{M}_2(m,n)$,
	  $$ \begin{array}{rcl} |B_{d-2}| &=& |C'_1| + (n+1)|C'_2| + (m+1)|C'_{d-1}|\mbox{ and }\\
				|B_{d-1}| &=& |C'_1| + n|C'_2| + |C'_{d}|
	     \end{array}$$
      and we may again choose appropriate $m$ and $n$ to satisfy our claim.
  \end{proof}

    We now assume $B_j$ represents the block associated to $I^{(N)}_j$, $1\leq j \leq d$, for the induced map on $S$ with permutation $\sigma_{(4213)}$. We then act by the path associated to $\tilde{M}_1(m,n)$ and call the new blocks $\hat{B}_j$. We show the following:

   \begin{lem}\label{C bounded}
   Let $B_j$ and $\hat{B}_j$, $1\leq j \leq d$, be given as above and consider the sequence associated to the IET with permutation $\sigma_{(2431)}$ after acting by $\tilde{M}_1(m,n)$. There exists a constant $C$ such that the gaps between blocks of one of the following forms: $B_{d-1}^n$ or $B_{d-2}^mB_{d-1}^n$, are of length at most $C$. This constant is independent of $m,n$.
   \end{lem}

  \begin{proof}
  This is given by direct calculation of how the words $\hat{B}_j$ are formed from words $B_i$. We list the forms by case.

  If $\sigma$ is \NICE, then
      $$ \begin{array}{lcl}
	      \hat{B}_{\ell} & = & B_\ell*B_{d-1}^{n}*B_d, \mbox{ for }1\leq \ell \leq d-4\\
	      \hat{B}_{d-3} &=& B_{d-3} B_{d-1}^{n+1} B_d\\
	      \hat{B}_{d-2} &=& B_{d-3} B_{d-1}^{n+1} B_d\\
	      \hat{B}_{d-1} &=& B_{d-2}^{m+1} B_{d-1}^{n+1} B_d\\
	      \hat{B}_{d} &= &B_{d-2}^m B_{d-1}^{n+1} B_d.
         \end{array}$$

  If $\sigma$ is \ANICE, then
      $$\begin{array}{lcl}
	    \hat{B}_1 &=& B_1 B_{d-1}^{n+1}B_d\\
	    \hat{B}_2 &=& B_1 B_{d-1}^n B_d\\
	    \hat{B}_\ell &= &B_{\ell-1}B_{d-1}^n B_d,\mbox{ for }3\leq \ell \leq d-2\\
	    \hat{B}_{d-1} &=& B^{m+1}_{d-2} B^{n+1}_{d-1} B_d\\
	    \hat{B}_d &=& B_{d-2}^{m} B_{d-1}^{n+1} B_d.
	\end{array}$$

   Because our sequence is composed of words of the form $\hat{B}_j$, the following $C$ will suffice
      $$ C= \max_{1\leq j \leq d-3}|B_j|+|B_{d-2}|+|B_{d-1}|+|B_d|$$
  by comparing all possible words of the form $\hat{B}_j\hat{B}_{j'}$.
  \end{proof}

  We use this result in Section \ref{other classes} by taking $\tilde{M}_1(m,n)$ with $m,n\gg C$. We end by remarking that in both cases, all words $\hat{B}_j$ take one of two forms, one that contains $B_{d-1}^n$ and one that contains $B_{d-2}^mB_{d-1}^n$. This is a direct result of the way the class $\RRR_4$ embeds into larger classes.

\end{document}